\documentclass[12pt]{amsart}

\usepackage{amsmath,amssymb,amsthm,mathtools}
\usepackage{enumitem,booktabs,microtype}
\usepackage[hidelinks]{hyperref}

\newtheorem{theorem}{Theorem}[section]
\newtheorem{proposition}[theorem]{Proposition}
\newtheorem{lemma}[theorem]{Lemma}
\newtheorem{corollary}[theorem]{Corollary}
\theoremstyle{definition}
\newtheorem{definition}[theorem]{Definition}
\theoremstyle{remark}
\newtheorem{remark}[theorem]{Remark}
\newtheorem{example}[theorem]{Example}

\newcommand{\R}{\mathbb R}
\newcommand{\Q}{\mathbb Q}
\newcommand{\cH}{\mathcal H}

\title[Newton Support Functions and Metric Completion]{Newton Support Functions and Metric Completion of Singular Conformal Metrics at Corners}

\author{Muhamad Fahmi bin Zanal Abidin}
\address{Independent Researcher, Kuala Lumpur, Malaysia}
\email{fahmiemail123@gmail.com}


\date{}

\begin{document}

\begin{abstract}
We study singular conformal metrics \(g=Fg_0\) near codimension-two corners,
where \(F\) is a finite positive sum of monomial singularities.  The competing
orders are encoded by the homogeneous Newton support function
\(\mathcal H(p,q)=\max_j(pa_j+qb_j)\).  We prove the sharp projective
accessibility criterion \(\mathcal H(p,q)<2\min\{p,q\}\), and deduce that the
corner lies at finite metric distance exactly when
\(\max_j(a_j+b_j)<2\).  At every accessible corner, all normal approaches
over a fixed corner point determine one canonical completion point.  The
induced boundary metric is locally bi-Lipschitz to a snowflake of the smooth
corner metric, yielding an explicit Hausdorff-dimension formula.  For rational
exponent data, the Newton decomposition is also realized by compatible rooted
and monoidal modifications.  The metric conclusions themselves require only
positivity, bounded coefficients, and uniform ellipticity.
\end{abstract}

\maketitle

\noindent
\textit{ORCID:} \href{https://orcid.org/0009-0009-1024-9858}{0009-0009-1024-9858}

\section{Introduction}

Metric completion of a singular Riemannian metric is governed by the cost of
approaching the singular set.  Near a corner, this cost can reflect several
competing asymptotic orders rather than a single radial exponent.  The basic
problem considered here is to identify a finite invariant that records this
competition and at the same time determines the resulting completion
geometry.

Newton polyhedra provide a classical way to organize competing monomial
orders in singular asymptotics; Varchenko's work on oscillatory integrals is a
standard example \cite{Varchenko1976}, while standard toric references record
the corresponding convex-geometric and fan constructions
\cite{Fulton1993,CoxLittleSchenck2011}.  Weighted and generalized blow-up
constructions provide a complementary geometric language for resolving
corners.  In particular, Kottke--Melrose encode generalized boundary blow-up
by refinements of monoidal data \cite{KottkeMelrose2015}, Kottke extends the
refinement/blow-up correspondence to generalized corners \cite{Kottke2018},
and Joyce develops the underlying generalized-corner category
\cite{Joyce2016}.

The broader geometric-analytic setting includes singular Riemannian spaces
\cite{Cheeger1983}, edge-degenerate geometry \cite{Mazzeo1991}, and analysis
on manifolds with corners \cite{Melrose1992}.  From a length-metric viewpoint,
Romney's conformal Grushin spaces give a closely related class of singular
conformal deformations \cite{Romney2016}.  General conformal-density theory
studies boundary size and accessibility for deformed Euclidean domains
\cite{BonkKoskela2002,Nieminen2009,KlenSuomala2013}, while more recent work on
conformally singular boundary metrics studies spectral and
Hausdorff-dimensional phenomena for single-boundary degeneracies
\cite{ColinDeVerdiereDietzeDeHoopTrelat2026}.  These works provide the
surrounding singular-metric context.  The present metric statements are
proved directly from curve-length estimates and comparison arguments, and
address the additional competition created by finitely many monomial corner
interactions.

Let \(C=\{x=y=0\}\) be a codimension-two corner and consider
\begin{equation}\label{eq:model}
 g=F\,g_0,\qquad
 F(x,y,z)=\sum_{j=1}^N u_j(x,y,z)x^{-a_j}y^{-b_j},
\end{equation}
where \(a_j,b_j\ge0\), the coefficients \(u_j\) are positive and bounded
above and below, and \(g_0\) is uniformly elliptic and smooth.  The exponent
set determines the homogeneous Newton support function
\begin{equation}\label{eq:support}
 \cH(p,q)=\max_{1\le j\le N}(pa_j+qb_j),\qquad p,q>0.
\end{equation}
The homogeneity of \(\cH\) makes
\(\cH(p,q)/\min\{p,q\}\) a function of the projective weighted direction.
The central point of the paper is that this support function is also the
relevant metric invariant.

The first main result gives the sharp directional criterion
\begin{equation}\label{eq:intro-direction}
 [(p,q)]\ \text{is accessible}
 \quad\Longleftrightarrow\quad
 \cH(p,q)<2\min\{p,q\}.
\end{equation}
For nonnegative exponent data, minimizing the normalized support over
positive directions occurs at the diagonal, and therefore the original
corner is accessible exactly when
\[
 A:=\cH(1,1)=\max_j(a_j+b_j)<2.
\]
The same diagonal order governs the geometry that survives in the metric
completion.  We prove that the full normal fibre over each accessible corner
point collapses to a single canonical completion point; if \(d_C\) is a
smooth metric on the corner stratum, then locally
\[
 d_{\partial}\asymp d_C^{\,1-A/2},
\]
and the corresponding Hausdorff dimension is
\(\dim C/(1-A/2)\).  Thus the Newton support function controls both approach
to the corner and the intrinsic geometry of the completed stratum.

A second theme is the distinction between metric and resolution hypotheses.
The metric theory allows arbitrary real exponents and bounded positive
coefficients.  Smooth coefficients are needed only when one asks for smooth
residual units after resolution, while rationality is needed only for the
rooted integral/monoidal realization.  The mixed two-monomial calculation
also shows why scalar monomialization alone is insufficient: the differential
of the weighted blow-down carries part of the metric scaling and must be
tracked simultaneously.

The results are deliberately limited to finite positive Newton support in
codimension two.  Positivity prevents cancellation between leading face
terms, while finiteness makes the accessibility set a finite projective fan.
Signed interactions, higher codimension, and arbitrary nonpolyhomogeneous
singular factors require additional ideas and are not claimed here.

The author's related work \cite{ZanalAbidinPaper1} treats the
single-interaction and orthogonal-support special cases of this framework.
The present paper is self-contained and extends those models to arbitrary
finite positive Newton supports, producing the projective accessibility
criterion, the complete accessibility fan, and the corresponding completion
and boundary-snowflake classification.

\paragraph{Organization.}
Section~\ref{sec:local-comparison} establishes the local comparison estimates.
Section~\ref{sec:newton-directions} develops the Newton support function and
proves the directional accessibility criterion.
Section~\ref{sec:mixed-benchmark} analyzes the mixed two-monomial model.
Sections~\ref{sec:corner-accessibility}--\ref{sec:boundary-metric} derive
corner accessibility, the finite accessibility fan, the canonical corner
embedding, and the boundary snowflake law.
Sections~\ref{sec:rational-resolution} and~\ref{sec:globalization} treat
rational geometric realization and conditional globalization.
Section~\ref{sec:sharpness} isolates the roles of the hypotheses.
The main theorem is collected in Section~\ref{sec:main-theorem}, followed by
the comparison with existing work in Section~\ref{sec:related-work} and the
discussion of limitations in Section~\ref{sec:limitations}.

\section{Local setting and elementary comparison}\label{sec:local-comparison}

Let
\[
U=[0,\varepsilon)^2\times V
\]
be an adapted coordinate neighborhood, where \(V\) is a relatively compact
coordinate patch in the corner stratum.  Write \(z\) for the tangential
variables.

\begin{definition}[Positive finite Newton interaction]
A function \(F\) on \(U^\circ\) is a positive finite Newton interaction if
\[
F=\sum_{j=1}^N u_jx^{-a_j}y^{-b_j},
\]
where \(N<\infty\), \(a_j,b_j\in\R_{\ge0}\), and
\[
0<c_j\le u_j\le C_j<\infty
\]
on \(U^\circ\).
\end{definition}

\begin{lemma}[Frozen-support comparison]\label{lem:frozen}
Set
\[
F_-=\sum_jc_jx^{-a_j}y^{-b_j},
\qquad
F_+=\sum_jC_jx^{-a_j}y^{-b_j}.
\]
Then
\[
F_-\le F\le F_+.
\]
If
\[
c_0g_{\rm E}\le g_0\le C_0g_{\rm E}
\]
on \(U\), then
\[
c_0F_-g_{\rm E}\le Fg_0\le C_0F_+g_{\rm E}.
\]
\end{lemma}

\begin{proof}
The scalar inequalities hold term by term.  Multiplication by a positive
quadratic form preserves order.
\end{proof}

\begin{lemma}[Stay-versus-exit localization]\label{lem:exit}
Let \(W\Subset U\).  Under the hypotheses above there is
\(\eta(W,U)>0\) such that every rectifiable curve starting in \(W\) and
leaving \(U\) has \(g\)-length at least \(\eta\).
\end{lemma}

\begin{proof}
Because the exponents are nonnegative and the box is bounded, \(F\) has a
positive lower bound on \(U^\circ\).  Uniform positive definiteness gives
\(g\ge c\,g_{\rm E}\).  The Euclidean distance from \(\overline W\) to
\(U^c\) is positive, so every exiting curve has a fixed positive Euclidean,
and hence \(g\)-, length.
\end{proof}

Lemma~\ref{lem:exit} is the localization step needed to transfer restricted
chart estimates to the ambient completion metric: sufficiently small local
boundary distances cannot be shortened by leaving the chart.

\subsection*{Completion convention and directional accessibility}
Let \((\widehat{U^\circ},\widehat d_g)\) denote the metric completion of
\((U^\circ,d_g)\), realized as equivalence classes of \(d_g\)-Cauchy
sequences.  A projective weighted direction \([(p,q)]\), with \(p,q>0\),
is called \emph{directionally accessible} if, after normalizing
\(m=q/p\), there exists a finite-length curve approaching \(r=0\) in the
weighted chart
\[
 x=r,\qquad y=r^m s,
\]
which eventually remains in a compact interval
\(0<s_-\le s\le s_+<\infty\).  This sector-local notion is distinct from
the classification of the full completion fibre over a background corner
point, which is established in Section~\ref{sec:corner-fibre}.

\section{Newton support and weighted directions}\label{sec:newton-directions}

For \(m>0\), define
\[
H(m)=\max_j(a_j+mb_j).
\]
Equivalently,
\[
H(m)=\cH(1,m).
\]

\begin{proposition}[Power-law support]\label{prop:power}
Along \(y=x^ms\), with \(s\) restricted to a compact interval
\(0<s_-\le s\le s_+<\infty\),
\[
F(x,x^ms,z)\asymp x^{-H(m)}
\]
uniformly in \(s\) and \(z\).
\end{proposition}

\begin{proof}
Choose \(j_*\) with \(a_{j_*}+mb_{j_*}=H(m)\).  The corresponding positive
term gives the lower bound.  Every other exponent satisfies
\(a_j+mb_j\le H(m)\), and the boundedness of \(s^{-b_j}\) and \(u_j\)
gives the upper bound.
\end{proof}

The weighted substitution
\begin{equation}\label{eq:weighted}
x=r,\qquad y=r^ms
\end{equation}
has the exact differential
\begin{equation}\label{eq:dyexact}
dy=mr^{m-1}s\,dr+r^m\,ds.
\end{equation}
For the Euclidean normal metric,
\begin{equation}\label{eq:exact-coframe}
dx^2+dy^2
=
dr^2+\bigl(mr^{m-1}s\,dr+r^m ds\bigr)^2.
\end{equation}
The second square is exactly \(dy^2\); it must not be discarded merely
because one coefficient appears lower order in one chart.

\begin{theorem}[Newton accessibility theorem]\label{thm:direction}
Let \(m>0\), and restrict to a compact interior angular set
\(s\in[s_-,s_+]\).  The direction is directionally accessible if and only if
\begin{equation}\label{eq:direction-criterion}
H(m)<2\min\{1,m\}.
\end{equation}
Equivalently, for a homogeneous direction \([(p,q)]\),
\begin{equation}\label{eq:homogeneous-criterion}
\boxed{\cH(p,q)<2\min\{p,q\}.}
\end{equation}
Equality is inaccessible.
\end{theorem}

\begin{proof}
By Lemma~\ref{lem:frozen} it suffices to work with the Euclidean background
and a frozen positive Newton sum.  Proposition~\ref{prop:power} gives
\(F\asymp r^{-H(m)}\).

For sufficiency, take \(s=s_0\) constant.  From
\eqref{eq:exact-coframe},
\[
|\dot\gamma|_g^2
\asymp
r^{-H(m)}
\left(1+m^2s_0^2r^{2m-2}\right)\dot r^2.
\]
If \(m\ge1\), the length is bounded by a constant multiple of
\[
\int_0^\varepsilon r^{-H(m)/2}\,dr,
\]
which converges exactly when \(H(m)<2\).  If \(0<m<1\), the second term
dominates and the length is bounded by
\[
\int_0^\varepsilon r^{-H(m)/2+m-1}\,dr,
\]
which converges exactly when \(H(m)<2m\).

For necessity, \eqref{eq:exact-coframe} gives the exact lower bounds
\[
g\ge c\,r^{-H(m)}\,dr^2,
\qquad
g\ge c\,r^{-H(m)}\,dy^2.
\]
For an arbitrary rectifiable approach curve \(\gamma(t)\), no monotonicity
of \(r(t)\) is needed.  Indeed, the one-dimensional coarea (equivalently,
level-crossing) formula gives
\[
\int \sqrt{h(r(t))}\,|r'(t)|\,dt
\ge
\int_0^{r_*}\sqrt{h(s)}\,ds,
\]
because every level \(s\in(0,r_*)\) is crossed at least once by a curve
approaching \(r=0\).  Applying this with \(h(r)=c r^{-H(m)}\), the first
lower bound yields divergence when \(H(m)\ge2\).

On the compact angular set \(y=r^ms\asymp r^m\), hence
\(r^{-H(m)}\asymp y^{-H(m)/m}\).  The same level-crossing argument,
now applied to the coordinate \(y\), shows that the second lower bound
yields divergence when \(H(m)/m\ge2\), i.e. when \(H(m)\ge2m\).
Combining them gives
\eqref{eq:direction-criterion}.  At equality the divergence is logarithmic:
if \(m\ge1\) and \(H(m)=2\), then the radial lower bound contains
\[
\int_0^\varepsilon r^{-H(m)/2}\,dr
=
\int_0^\varepsilon \frac{dr}{r}
=
\infty.
\]
If \(0<m<1\) and \(H(m)=2m\), then
\[
-\frac{H(m)}2+m-1=-1,
\]
so the corresponding lower-bound integral is again
\(\int_0^\varepsilon r^{-1}\,dr=\infty\).

For \((p,q)\), normalize by \(p\): \(m=q/p\).  Homogeneity gives
\[
\cH(p,q)=pH(q/p),
\]
and multiplication of the normalized criterion by \(p\) yields
\eqref{eq:homogeneous-criterion}.
\end{proof}

\section{The mixed two-monomial benchmark}\label{sec:mixed-benchmark}

The first genuinely mixed interaction is
\begin{equation}\label{eq:mixed}
F=x^{-a_1}y^{-b_1}+x^{-a_2}y^{-b_2}.
\end{equation}
Assume, after relabeling if necessary,
\[
a_1>a_2,\qquad b_1<b_2.
\]
Then the exponent vectors are incomparable coordinatewise and there is a
unique positive balance slope
\begin{equation}\label{eq:kappa}
\kappa=\frac{a_1-a_2}{b_2-b_1}>0.
\end{equation}
Set
\begin{equation}\label{eq:lambda}
\lambda=a_1+\kappa b_1=a_2+\kappa b_2.
\end{equation}

Under
\[
x=r,\qquad y=r^\kappa s,
\]
the scalar factor resolves exactly:
\begin{equation}\label{eq:mixed-factor}
F=r^{-\lambda}\bigl(s^{-b_1}+s^{-b_2}\bigr).
\end{equation}
The quotient of the two terms is
\[
s^{\,b_2-b_1},
\]
so \(s\to0\) and \(s\to\infty\) recover the two pure dominance regimes,
while compact \(s\)-intervals describe the balance face.

\begin{theorem}[Mixed balance-face criterion]\label{thm:mixed}
On every compact interior part of the balance face, the resolved direction
is accessible if and only if
\[
\boxed{\lambda<2\min\{1,\kappa\}.}
\]
\end{theorem}

\begin{proof}
This is Theorem~\ref{thm:direction} at the balancing slope
\(m=\kappa\).  Directly, the exact pullback is
\[
r^{-\lambda}A(s)
\left[
dr^2+
\bigl(\kappa r^{\kappa-1}s\,dr+r^\kappa ds\bigr)^2
\right],
\]
where \(A(s)=s^{-b_1}+s^{-b_2}\) is bounded above and below on compact
interior angular sets.  The two exact squares yield the two necessary
barriers \(\lambda<2\) and \(\lambda<2\kappa\), while the constant-\(s\)
path proves sufficiency.
\end{proof}

This benchmark exhibits the key methodological point: scalar
monomialization alone does not determine the pulled-back metric.  The
blow-down differential carries part of the singular scaling.

\section{Global corner accessibility}\label{sec:corner-accessibility}

Set
\begin{equation}\label{eq:A}
A=\max_j(a_j+b_j)=H(1).
\end{equation}

\begin{theorem}[Corner accessibility criterion]\label{thm:corner}
The codimension-two corner \(C=\{x=y=0\}\) lies at finite distance from the
interior if and only if
\begin{equation}\label{eq:cornercrit}
\boxed{A<2.}
\end{equation}
\end{theorem}

\begin{proof}
If \(A<2\), take the diagonal path \(x=y=r\), with \(z\) fixed.
Then
\[
F(r,r,z)\asymp r^{-A},
\]
and the length is bounded by
\[
C\int_0^\varepsilon r^{-A/2}\,dr<\infty.
\]

Conversely, choose \(j_*\) with \(a_{j_*}+b_{j_*}=A\).  Positivity gives
\[
g\ge c\,x^{-a_{j_*}}y^{-b_{j_*}}(dx^2+dy^2).
\]
Let
\[
\rho=(x^2+y^2)^{1/2}.
\]
Since \(x,y\le\rho\) and \(a_{j_*},b_{j_*}\ge0\),
\[
x^{-a_{j_*}}y^{-b_{j_*}}\ge\rho^{-A}.
\]
Also \(|d\rho|\le (dx^2+dy^2)^{1/2}\).  Hence every curve approaching the
corner satisfies
\[
L_g\ge c\int \rho^{-A/2}|d\rho|.
\]
This diverges for \(A\ge2\), including logarithmically at \(A=2\).
\end{proof}

\begin{remark}
Theorem~\ref{thm:corner} is a global accessibility statement about the
corner as a completion stratum.  Theorem~\ref{thm:direction} is finer:
even when \(A<2\), some resolved weighted directions can remain at infinite
distance.
\end{remark}

\section{Finite accessibility fan}\label{sec:accessibility-fan}

Because \(N<\infty\),
\[
H(m)=\max_j(a_j+mb_j)
\]
is continuous, convex and piecewise affine with finitely many breakpoints.

\begin{theorem}[Finite accessibility fan]\label{thm:fan}
The set
\[
\mathcal F_{\mathrm{acc}}
=
\{m>0:H(m)<2\min\{1,m\}\}
\]
is an open subset of \((0,\infty)\) with finitely many connected
components.  Its endpoints occur among:
\begin{enumerate}[label=(\roman*)]
\item Newton breakpoints where two affine orders agree;
\item solutions of \(a_j+mb_j=2m\) on \(0<m\le1\);
\item solutions of \(a_j+mb_j=2\) on \(m\ge1\);
\item the normalization point \(m=1\).
\end{enumerate}
\end{theorem}

\begin{proof}
The maximum of finitely many affine functions is continuous and piecewise
affine with finitely many pieces.  Subtracting
\(2\min\{1,m\}\), itself piecewise affine with one breakpoint, preserves
that property.  A strict negative set of a continuous piecewise-affine
function with finitely many pieces is open with finitely many components,
and its finite endpoints occur where one of the active affine equalities
listed above is attained.
\end{proof}

\section{Canonical corner embedding and completion fibre}
\label{sec:corner-fibre}

Assume throughout this section that
\[
 A=\max_j(a_j+b_j)<2,
 \qquad
 \delta=1-\frac A2>0.
\]
Let \(C\) be a connected local piece of the corner stratum.  For
\(z\in C\) and \(r>0\), set
\[
 p_r(z)=(r,r,z).
\]

\begin{lemma}[Diagonal Cauchy representatives]
\label{lem:diagonal-cauchy}
For each \(z\in C\), the family \(p_r(z)\) is Cauchy as \(r\downarrow0\).
More precisely, for \(0<r<s\),
\[
 d_g\bigl(p_r(z),p_s(z)\bigr)
 \le C\bigl(s^\delta-r^\delta\bigr).
\]
Consequently the class
\[
 \iota_C(z):=[p_r(z)]_{r\downarrow0}
 \in\widehat{U^\circ}
\]
is well defined and independent of the chosen sequence \(r\downarrow0\).
\end{lemma}

\begin{proof}
Join the two points along the normal diagonal.  On this path
\(F(t,t,z)\le C t^{-A}\), while the smooth background metric is uniformly
bounded above by a Euclidean metric.  Hence the length is bounded by
\[
 C\int_r^s t^{-A/2}\,dt
 =C'\bigl(s^\delta-r^\delta\bigr).
\]
The same estimate applied to two arbitrary null sequences proves
independence of the representative.
\end{proof}

\begin{lemma}[Uniform normal reduction]
\label{lem:uniform-normal-reduction}
There is a constant \(C>0\) such that, for every interior point
\(P=(x,y,z)\) sufficiently near the corner and
\(R=\max\{x,y\}\),
\[
 d_g\bigl(P,p_R(z)\bigr)\le C R^\delta.
\]
Consequently,
\[
 \widehat d_g\bigl(P,\iota_C(z)\bigr)\le C R^\delta.
\]
\end{lemma}

\begin{proof}
Suppose first that \(x\le y=R\).  Keep \(y=R\) and \(z\) fixed and
increase \(x=t\) from \(x\) to \(R\).  Uniform ellipticity and the upper
coefficient bounds give
\[
 |\dot\gamma|_g
 \le C\sum_{j=1}^N t^{-a_j/2}R^{-b_j/2}|\dot t|.
\]
Since \(a_j\le a_j+b_j\le A<2\), integration yields
\[
 L_g(\gamma)
 \le C\sum_j R^{-b_j/2}
       \int_x^R t^{-a_j/2}\,dt
 \le C\sum_j R^{1-(a_j+b_j)/2}
 \le C R^\delta.
\]
The case \(y\le x=R\) is identical, with the roles of \(x\) and \(y\)
interchanged.  This proves the first estimate.  The second follows by
joining \(p_R(z)\) to its diagonal completion class and using
Lemma~\ref{lem:diagonal-cauchy}.
\end{proof}

\begin{theorem}[Canonical normal-fibre collapse]
\label{thm:canonical-corner-fibre}
If \(A<2\), the map
\[
 \iota_C:C\longrightarrow\widehat{U^\circ}
\]
defined above is canonical.  Every interior sequence
\[
 P_n=(x_n,y_n,z)
\]
with fixed tangential label \(z\) and \(x_n,y_n\to0\) converges to
\(\iota_C(z)\).  In particular, all directionally accessible weighted
sectors over the same background point, even when they lie in different
connected components of the accessibility fan, determine the same
completion point.  No additional normal branch occurs over a fixed
background label \(z\).
\end{theorem}

\begin{proof}
The assertion follows immediately from
Lemma~\ref{lem:uniform-normal-reduction}, because
\(R_n=\max\{x_n,y_n\}\to0\).  This estimate is independent of the slope
of approach and does not require the sequence to remain in a single
weighted chart.  Therefore every directionally accessible sector over
\(z\) has the same limit \(\iota_C(z)\), regardless of intervening
inaccessible sectors in the projective accessibility fan.
\end{proof}

\begin{example}[A three-term Newton interaction]\label{ex:three-term}
Consider
\[
F(x,y)=x^{-1/2}+x^{-1/5}y^{-6/5}+y^{-3/2}.
\]
The exponent vectors are
\[
\left(\frac12,0\right),\qquad
\left(\frac15,\frac65\right),\qquad
\left(0,\frac32\right),
\]
and hence
\[
H(m)
=
\max\left\{
\frac12,\,
\frac15+\frac65m,\,
\frac32m
\right\}.
\]
The first and second affine orders agree at \(m=1/4\), while the second and
third agree at \(m=2/3\).  Therefore
\[
H(m)=
\begin{cases}
\frac12, & 0<m\le \frac14,\\[3pt]
\frac15+\frac65m, & \frac14\le m\le\frac23,\\[3pt]
\frac32m, & m\ge\frac23.
\end{cases}
\]

For \(0<m<1\), accessibility is equivalent to \(H(m)<2m\).  On the first
piece this requires \(m>1/4\); on the middle and third pieces it holds
strictly for every \(m>1/4\).  For \(m\ge1\), accessibility is equivalent to
\(H(m)<2\), which on the last piece gives \(m<4/3\).  The accessibility fan
is therefore the single interval
\[
\mathcal F_{\mathrm{acc}}
=
\left(\frac14,\frac43\right),
\]
with logarithmic divergence at both endpoints.

The diagonal order is
\[
A=H(1)=\frac32<2,
\]
so the corner is accessible.  If the corner stratum has dimension \(k\),
Theorem~\ref{thm:snowflake} and Corollary~\ref{cor:hd} give
\[
d_{\partial}\asymp d_C^{1/4},
\qquad
\dim_H(C,d_{\partial})=4k.
\]
Thus one explicit Newton polygon simultaneously determines the accessible
weighted sectors, the corner threshold, the boundary snowflake exponent, and
the Hausdorff dimension.
\end{example}

\section{Boundary metric and collapse}\label{sec:boundary-metric}

Assume \(A<2\), let \(C\) and \(\iota_C\) be as in
Section~\ref{sec:corner-fibre}, and let \(d_C\) be any smooth Riemannian
distance induced by the restriction of \(g_0\) to \(TC\).  Define
\[
 d_{\partial}(z_0,z_1)
 =\widehat d_g\bigl(\iota_C(z_0),\iota_C(z_1)\bigr).
\]

\begin{theorem}[Boundary snowflake theorem]\label{thm:snowflake}
For sufficiently close \(z_0,z_1\in C\),
\begin{equation}\label{eq:snowflake}
\boxed{
d_{\partial}(z_0,z_1)\asymp
d_C(z_0,z_1)^{\,1-A/2}.
}
\end{equation}
\end{theorem}

\begin{proof}
Write
\[
\delta=d_C(z_0,z_1).
\]
For the upper bound, move normally from the corner to the diagonal level
\(x=y=r\), move tangentially there, and return normally.  The two normal
legs cost \(O(r^{1-A/2})\).  On the diagonal, \(F\asymp r^{-A}\), so the
tangential leg costs \(O(r^{-A/2}\delta)\).  Therefore
\[
d_{\partial}(z_0,z_1)
\le
C\left(r^{1-A/2}+r^{-A/2}\delta\right).
\]
Choosing \(r\asymp\delta\) gives
\[
d_{\partial}(z_0,z_1)\le C\delta^{1-A/2}.
\]

For the lower bound, let \(\gamma\) be any sufficiently short curve joining
near-corner representatives of \(z_0,z_1\), and let
\[
R=\max_\gamma \rho,\qquad
\rho=(x^2+y^2)^{1/2}.
\]
If the representatives lie on the level \(\rho=\varepsilon\), the radial
barrier from the proof of Theorem~\ref{thm:corner} gives
\[
L_g(\gamma)\ge c\bigl(R^{1-A/2}-\varepsilon^{1-A/2}\bigr).
\]
On the portion of the curve with \(\rho\le R\), positivity of a maximizing
term gives \(F\ge cR^{-A}\).  Since the tangential projection must connect
\(z_0\) to \(z_1\), its \(g_0\)-length is at least \(c\delta\), and hence
\[
L_g(\gamma)\ge cR^{-A/2}\delta.
\]
Thus
\[
L_g(\gamma)\ge
c\max\{R^{1-A/2},R^{-A/2}\delta\}.
\]
Minimizing the right-hand side over \(R>0\) occurs at \(R\asymp\delta\),
yielding
\[
L_g(\gamma)\ge c\delta^{1-A/2}.
\]
To pass from interior representatives to completion points, take
representatives with \(\rho=\varepsilon\), apply the preceding estimates
with the harmless endpoint error \(O(\varepsilon^{1-A/2})\), and then let
\(\varepsilon\downarrow0\).  Lemma~\ref{lem:exit} rules out a cheaper
ambient shortcut outside the local chart for sufficiently small
\(\delta\).
\end{proof}

\begin{corollary}[Canonical corner embedding]
\label{cor:canonical-corner-embedding}
The map
\[
 \iota_C:(C,d_C)\longrightarrow\widehat{U^\circ}
\]
is injective.  Its image is a canonical copy of the corner stratum, and the
induced metric is locally bi-Lipschitz equivalent to
\(d_C^{\,1-A/2}\).
\end{corollary}

\begin{proof}
The lower bound in Theorem~\ref{thm:snowflake} is strictly positive for
distinct sufficiently close labels.  Local injectivity therefore follows,
and injectivity on a connected coordinate patch follows after shrinking the
patch if necessary.  The bi-Lipschitz statement is exactly
Theorem~\ref{thm:snowflake}.
\end{proof}

\begin{corollary}[Hausdorff dimension]\label{cor:hd}
If the smooth corner stratum has dimension \(k\), then locally
\[
\boxed{
\dim_H(C,d_{\partial})
=
\frac{k}{1-A/2}.
}
\]
\end{corollary}

\begin{proof}
A \(k\)-dimensional smooth Riemannian metric is locally Ahlfors
\(k\)-regular.  Replacing its distance \(d_C\) by the snowflake
\(d_C^\theta\), \(0<\theta\le1\), rescales Hausdorff dimension by
\(1/\theta\), while bi-Lipschitz equivalence preserves Hausdorff
dimension; see, for example, \cite{BBI2001,Heinonen2001}.  Apply
Theorem~\ref{thm:snowflake} with \(\theta=1-A/2\).
\end{proof}

\begin{corollary}[Resolved-sector collapse]
\label{prop:collapse}
Fix an accessible weighted direction \(m\) and a compact interval
\(I\Subset(0,\infty)\) in its angular variable \(s\).  All points
\((r,s,z_0)\) with \(s\in I\) converge to the same completion point as
\(r\downarrow0\), for fixed \(z_0\).
\end{corollary}

\begin{proof}
At fixed \(r\), the angular segment has
\[
dy=r^m\,ds.
\]
Using \(F\asymp r^{-H(m)}\), its length is
\[
O\!\left(r^{m-H(m)/2}|s_1-s_2|\right).
\]
Accessibility implies \(H(m)<2m\), so this tends to zero.  Hence the diameter of every compact interior angular interval tends to zero
in the completion.  By Theorem~\ref{thm:canonical-corner-fibre}, its limit
is the canonical point \(\iota_C(z_0)\), and the same conclusion holds
across all accessible components of the fan.
\end{proof}

\begin{theorem}[Classification of the local corner fibre]
\label{thm:corner-cauchy-classification}
Let \(W\Subset U\), and let
\[
 P_n=(x_n,y_n,z_n)\in W^\circ
\]
be a \(d_g\)-Cauchy sequence whose background normal coordinates satisfy
\(x_n,y_n\to0\).  Then \(z_n\) converges to a unique point \(z_*\in C\),
and
\[
 P_n\longrightarrow\iota_C(z_*)
\]
in the metric completion.  Thus the completion fibre over \(z_*\) is a
singleton.
\end{theorem}

\begin{proof}
By Lemma~\ref{lem:exit}, sufficiently short connections between sufficiently
late terms cannot leave \(U\).  On \(U\), the lower coefficient bound and
uniform ellipticity imply
\[
 g\ge c\,g_{\rm E}.
\]
Hence every sufficiently late tail is Euclidean Cauchy, so \(z_n\to z_*\)
for a unique \(z_*\).

Set \(R_n=\max\{x_n,y_n\}\).  Lemma~\ref{lem:uniform-normal-reduction}
gives
\[
 \widehat d_g\bigl(P_n,\iota_C(z_n)\bigr)\le C R_n^\delta\to0.
\]
The upper estimate in Theorem~\ref{thm:snowflake} gives
\[
 \widehat d_g\bigl(\iota_C(z_n),\iota_C(z_*)\bigr)
 \le C d_C(z_n,z_*)^\delta\to0.
\]
The triangle inequality proves the claim.  Uniqueness of the fibre follows
from injectivity of \(\iota_C\).
\end{proof}

\section{Rational Newton data and geometric resolution}\label{sec:rational-resolution}

The metric results above allow arbitrary real exponents.  We now impose
rationality only to obtain an integral geometric realization.

Assume \(a_j,b_j\in\Q_{\ge0}\).  Choose \(D\in\mathbb N\) such that
\[
A_j=Da_j,\qquad B_j=Db_j
\]
are integers.  Introduce rooted boundary variables
\begin{equation}\label{eq:root}
x=X^D,\qquad y=Y^D.
\end{equation}

\begin{lemma}[Root-chart compatibility]\label{lem:root}
If two positive defining functions satisfy
\[
x'=e^\phi x,
\]
then the rooted variables satisfy
\[
X'=e^{\phi/D}X.
\]
Hence a fixed common denominator defines compatible smooth rooted boundary
coordinates.
\end{lemma}

\begin{proof}
Raise the asserted identity to the \(D\)-th power.  The multiplier
\(e^{\phi/D}\) is smooth and positive, with smooth inverse.
\end{proof}

\begin{remark}
The root map \eqref{eq:root} is not being identified with an ordinary
generalized boundary blow-up.  It changes the smooth boundary structure.
Generalized blow-up theory \cite{KottkeMelrose2015,Kottke2018} is invoked
only after denominator clearing, for the subsequent monoidal refinement.
\end{remark}

On the rooted structure,
\[
F=\sum_j \widetilde u_jX^{-A_j}Y^{-B_j}.
\]
Write
\[
\alpha_j=(A_j,B_j)\in\mathbb Z_{\ge0}^2,
\qquad
P=\operatorname{conv}\{\alpha_1,\ldots,\alpha_N\}
\subset\mathbb R^2.
\]
For a face \(Q\subset P\), define its \emph{maximizing normal cone} by
\[
N_P^+(Q)
=
\left\{
\xi\in\mathbb R^2:
\langle \xi,\alpha\rangle
=
\max_{\beta\in P}\langle \xi,\beta\rangle
\text{ for every }\alpha\in Q
\right\}.
\]
The collection of these cones is the maximizing normal fan
\(\mathcal N^+(P)\).  Only its intersection with the positive quadrant is
relevant to boundary blow-down directions.  Repeated exponent points and
nonvertex points do not change \(P\) or \(\mathcal N^+(P)\), although their
positive coefficients remain part of the residual unit.

Let \(\Sigma\) be a finite regular rational fan in the positive quadrant
refining \(\mathcal N^+(P)\).  In a regular two-dimensional cone with primitive
generators
\[
v=(p,q),\qquad w=(p',q'),
\qquad |\det(v,w)|=1,
\]
use the monomial chart
\[
X=r^ps^{p'},\qquad
Y=r^qs^{q'}.
\]
Then
\[
X^{-A_j}Y^{-B_j}
=
r^{-(pA_j+qB_j)}
s^{-(p'A_j+q'B_j)}.
\]

\begin{theorem}[Positive monomial factorization on a regular Newton chart]
\label{thm:toric}
On each regular cone \(\sigma\in\Sigma\) contained in a maximizing normal cone \(N_P^+(Q)\),
\[
F=r^{-M}s^{-N}U(r,s,z),
\]
where \(M,N\in\mathbb Z_{\ge0}\) and \(U\) is smooth and strictly positive
up to the resolved faces.
\end{theorem}

\begin{proof}
Because \(\sigma\) is contained in a maximizing normal cone of \(P\), there is a face \(Q\subset P\).  Choose a vertex \(\alpha_*=(A_*,B_*)\) of \(Q\).  It maximizes the pairing with every vector in \(\sigma\).  In particular, for its primitive
generators \(v,w\),
\[
M=\langle v,\alpha_*\rangle=\max_j\langle v,\alpha_j\rangle,
\qquad
N=\langle w,\alpha_*\rangle=\max_j\langle w,\alpha_j\rangle.
\]
Factoring \(r^{-M}s^{-N}\) gives
\[
F=r^{-M}s^{-N}
\sum_j \widetilde u_j\,
r^{M-\langle v,\alpha_j\rangle}
s^{N-\langle w,\alpha_j\rangle}.
\]
All displayed residual exponents are nonnegative integers.  The
\(\alpha_*\)-term has residual exponent \((0,0)\), so its coefficient
extends as a strictly positive smooth function.  Every other term is a
smooth nonnegative monomial times a smooth positive coefficient.  Their
finite sum is therefore a smooth strictly positive unit up to the resolved
faces.
\end{proof}

In dimension two, a finite rational fan can be regularized by repeatedly
subdividing nonsmooth cones; this is the standard toric regularization step
(see, for example, \cite[\S2.6]{Fulton1993}).  After denominator clearing it
is the regular-fan step used here.  Once such a smooth refinement has
been chosen, Kottke--Melrose realize smooth refinements by generalized
boundary blow-up in the manifold-with-corners setting
\cite{KottkeMelrose2015}.  Kottke proves the corresponding
refinement-to-blow-up construction for generalized corners
\cite{Kottke2018}.  Joyce supplies the generalized-corner category used in
that framework \cite{Joyce2016}.

\section{Globalization}\label{sec:globalization}

Let \(X\) now be compact with embedded boundary faces.  In addition to
overlap invariance of the exponent sets, assume that after one common root
modification the local maximizing normal fans define refinements of the
boundary monoidal complex that admit a finite compatible global rational
refinement.  This is the exact compatibility condition needed by the
generalized blow-up construction.

\begin{proposition}[Invariance under defining-function changes]
\label{prop:defining}
If
\[
x'=e^\phi x,\qquad y'=e^\psi y,
\]
then
\[
u_jx^{-a_j}y^{-b_j}
=
u'_j(x')^{-a_j}(y')^{-b_j},
\qquad
u'_j=u_je^{a_j\phi+b_j\psi}.
\]
Thus the exponent vectors are unchanged.
\end{proposition}

\begin{proof}
Substitute \(x=e^{-\phi}x'\) and \(y=e^{-\psi}y'\).  On a relatively
compact coordinate neighborhood, smoothness of \(\phi\) and \(\psi\)
implies that \(e^{a_j\phi+b_j\psi}\) is bounded above and below by
positive constants.  Hence \(u'_j\) satisfies the same two-sided
admissibility bounds as \(u_j\).
\end{proof}

\begin{theorem}[Conditional global rational realization]
\label{thm:global}
Assume \(X\) is compact with embedded boundary faces, the positive Newton
exponent sets are finite and rational, and the associated local maximizing
normal fans admit a finite compatible global rational refinement of the
rooted boundary monoidal complex.  This is a conditional local-to-global
realization statement: existence of the compatible global refinement is an
explicit hypothesis, not a consequence of local exponent invariance alone.
Then there is a common denominator
\(D\), a compatible rooted corner structure \(X^{(D)}\), and a finite
regular monoidal refinement
\[
\widetilde X\longrightarrow X^{(D)}
\]
such that every resolved codimension-two chart has the positive monomial
normal form of Theorem~\ref{thm:toric}.  The projective accessibility
criterion \eqref{eq:homogeneous-criterion}, corner criterion
\eqref{eq:cornercrit}, collapse statement, and local snowflake law are
independent of the chosen positive defining functions.
\end{theorem}

\begin{proof}
Compactness gives finitely many local exponent sets, hence one common
denominator.  Lemma~\ref{lem:root} glues the corresponding root structures,
and Proposition~\ref{prop:defining} shows that changes of positive defining
functions preserve both the exponent vectors and the admissible coefficient
class.

By hypothesis, the rooted local maximizing normal fans admit a finite
compatible global rational refinement of the boundary monoidal complex.
Regularize this refinement cone by cone in dimension two.  The resulting
smooth refinement is realized by generalized monoidal blow-up using the
published frameworks \cite{KottkeMelrose2015,Kottke2018}.  On each resolved
chart, Theorem~\ref{thm:toric} gives the required positive monomial normal
form.

Finally, the metric statements are independent of the chosen positive
defining functions by Lemma~\ref{lem:frozen}, and independent of the
particular smooth uniformly elliptic background representative by
quadratic-form comparison.
\end{proof}

\begin{remark}[Local invariance versus global fan compatibility]
Proposition~\ref{prop:defining} proves that the exponent data are invariant
under positive changes of defining functions.  It does not by itself imply
that arbitrarily chosen local subdivisions glue to a global monoidal
refinement.  The compatibility hypothesis in Theorem~\ref{thm:global}
records precisely this additional global requirement.
\end{remark}

\section{Sharpness and scope of the hypotheses}\label{sec:sharpness}

The theorem package contains assumptions serving different purposes.  This
section records which are genuinely metric and which belong only to the
chosen smooth/global realization.

\begin{proposition}[Cancellation]\label{prop:cancellation}
The raw exponent-support law can fail for signed leading terms even when
\(F\) remains strictly positive.
\end{proposition}

\begin{proof}
Take
\[
F=(x^{-1}-y^{-1})^2+1
=x^{-2}+y^{-2}-2x^{-1}y^{-1}+1.
\]
The unsigned raw support gives order \(2\) on the diagonal, but
\(F(r,r)=1\).  Thus cancellation lowers the true singular order.
\end{proof}

\begin{proposition}[Coefficient degeneration changes the effective order]
If the lower or upper positive bound on a nominal unit is removed, the
effective exponent may change.
\end{proposition}

\begin{proof}
A vanishing coefficient gives
\[
x\cdot x^{-2}=x^{-1},
\]
whereas an unbounded coefficient gives
\[
x^{-1}\cdot x^{-1}=x^{-2}.
\]
\end{proof}

\begin{proposition}[Rationality is geometric, not metric]
Theorem~\ref{thm:direction}, Theorem~\ref{thm:corner}, and
Theorem~\ref{thm:snowflake} require no rationality of the exponents.
A common denominator \(D\) exists for a finite exponent set if and only if
all exponents are rational.
\end{proposition}

\begin{proof}
The metric proofs use only real powers on the open quadrant and two-sided
comparisons.  For the denominator statement, if \(Da\in\mathbb Z\), then
\(a=(Da)/D\in\mathbb Q\); conversely the least common multiple of finitely
many rational denominators clears all of them.
\end{proof}

\begin{proposition}[Nonnegativity and the diagonal minimum]
\label{prop:diagmin}
If \(a_j,b_j\ge0\), then
\begin{equation}\label{eq:diagmin}
\inf_{m>0}\frac{H(m)}{\min\{1,m\}}=H(1)=A.
\end{equation}
The conclusion can fail if negative exponents are allowed.
\end{proposition}

\begin{proof}
For \(m\ge1\), each \(a_j+mb_j\) is nondecreasing, so
\(H(m)\ge H(1)\).  For \(0<m\le1\),
\[
\frac{H(m)}m
=
\max_j\left(\frac{a_j}{m}+b_j\right)
\ge
\max_j(a_j+b_j)
=
H(1).
\]
Equality is attained at \(m=1\).  If negative exponents are allowed, take
\((a,b)=(2,-1)\): then \(H(m)=2-m\) decreases for \(m\ge1\).
\end{proof}

\begin{proposition}[Background nondegeneracy is essential for invariance]
If the background tensor is allowed to degenerate at the corner, the
accessibility threshold can change.
\end{proposition}

\begin{proof}
Let
\[
F=x^{-2}+y^{-2}.
\]
With Euclidean background, the diagonal has logarithmically infinite
length.  With
\[
g_0=x^2dx^2+y^2dy^2,
\]
the restriction to \(x=y=r\) satisfies
\[
Fg_0=4\,dr^2,
\]
so the same diagonal has finite length.
\end{proof}

\begin{remark}[Compactness, embeddedness, and codimension]
Compactness is a sufficient finite-type condition for the global theorem,
not a local metric hypothesis.  Embedded boundary faces belong to the
global monoidal framework used here, not to the local length calculation.
Codimension two is a scope condition of the present proof: in higher
codimension the projective direction space and fan are genuinely
higher-dimensional, and the one-dimensional seam ordering used here does
not constitute a proof.
\end{remark}

\section{Main theorem}\label{sec:main-theorem}

\begin{theorem}[Metric completion for positive finite Newton interactions]
\label{thm:main}
Let \(C\) be a codimension-two corner and suppose locally
\[
g=
\left(
\sum_{j=1}^N u_jx^{-a_j}y^{-b_j}
\right)g_0,
\]
where \(N<\infty\), \(a_j,b_j\in\R_{\ge0}\),
\(0<c_j\le u_j\le C_j\), and \(g_0\) is uniformly comparable to a smooth
Euclidean coordinate metric.  Define
\[
\cH(p,q)=\max_j(pa_j+qb_j),
\qquad
A=\cH(1,1).
\]
Then:
\begin{enumerate}[label=(\roman*)]
\item a positive weighted direction \([(p,q)]\) is accessible exactly when
\[
\cH(p,q)<2\min\{p,q\};
\]
\item the original corner is accessible exactly when \(A<2\);
\item if \(A<2\), every normal approach over a fixed tangential label
converges to a unique canonical completion point, and the resulting map
\(\iota_C:C\to\widehat{U^\circ}\) is injective;
\item if \(A<2\), then locally
\[
d_{\partial}\asymp d_C^{\,1-A/2};
\]
\item if the smooth corner stratum has dimension \(k\), then
\[
\dim_H(C,d_{\partial})=\frac{k}{1-A/2};
\]
\item if, in addition, the coefficients are smooth and the exponents are
rational, then locally, after denominator clearing and a finite regular
subdivision of the maximizing normal fan, the interaction admits smooth
positive monomial normal forms on resolved charts.  A global finite monoidal
realization follows under the additional compatibility hypothesis of
Theorem~\ref{thm:global}.
\end{enumerate}
\end{theorem}

\begin{proof}
Items (i)--(v) are respectively
Theorem~\ref{thm:direction}, Theorem~\ref{thm:corner},
Theorem~\ref{thm:canonical-corner-fibre} together with
Corollary~\ref{cor:canonical-corner-embedding},
Theorem~\ref{thm:snowflake}, and Corollary~\ref{cor:hd}.
Theorem~\ref{thm:corner-cauchy-classification} gives the corresponding
local Cauchy-sequence classification.  Item (vi) follows locally from Theorem~\ref{thm:toric} together with
regularization of the rational fan.  The corresponding global statement is
Theorem~\ref{thm:global}, under its additional compatibility hypothesis.
\end{proof}

\section{Relation to existing work}\label{sec:related-work}

The following distinction is important for the scope of the claims.

\paragraph{Established resolution theory.}
Kottke--Melrose construct generalized boundary blow-ups from refinements of
the basic monoidal complex of a manifold with corners
\cite{KottkeMelrose2015}.  Kottke proves the analogous
refinement-to-blow-up statement for manifolds with generalized corners
\cite{Kottke2018}, and Joyce develops the generalized-corner category used
in that setting \cite{Joyce2016}.  For discrete quasihomogeneous weights,
Behr gives a detailed invariant treatment of quasihomogeneous blow-ups of
\(p\)-submanifolds in terms of filtrations of the vanishing ideal
\cite{Behr2021}.  These works justify or contextualize the geometric
realization of weighted refinements; they are not cited as sources for the
metric accessibility or boundary-snowflake theorems proved here.

\paragraph{Established Newton methodology.}
Newton polyhedra are classical devices for organizing competing monomial
orders in singular asymptotic problems; Varchenko's oscillatory-integral
analysis is a standard example \cite{Varchenko1976}.  Newton-polyhedral
filtrations have also been used to construct anisotropic singular Riemannian
metrics in singularity theory, notably in the study of bi-Lipschitz and
differentiable sufficiency of jets \cite{SoaresCostaSaia2026}.  That use of
Newton data is directed toward finite determinacy and controlled vector-field
estimates; the present use is different, employing Newton support data to
determine intrinsic accessibility, completion fibres, and boundary metric
geometry.  The present paper therefore uses established Newton methodology
without attributing the metric-completion criterion itself to that literature.

\paragraph{Established singular and metric geometry.}
Cheeger's singular Riemannian spaces \cite{Cheeger1983}, Mazzeo's edge
calculus \cite{Mazzeo1991}, and Melrose's analysis on manifolds with corners
\cite{Melrose1992} form part of the wider geometric-analytic background.
Romney's conformal Grushin spaces \cite{Romney2016} and the recent singular
boundary metrics studied by Colin de Verdi\`ere, Dietze, de Hoop, and
Tr\'elat \cite{ColinDeVerdiereDietzeDeHoopTrelat2026} are closer in their
use of singular conformal length structures.  Navarro--Pan's cone--Grushin
examples provide another nearby singular-metric phenomenon, exhibiting
large Hausdorff dimension and inhomogeneous metric dilations
\cite{NavarroPan2026}.  General conformal-density work studies
metric boundaries, boundary accessibility, and Hausdorff or packing
dimensions for conformally deformed Euclidean domains; see Bonk--Koskela
\cite{BonkKoskela2002}, Nieminen \cite{Nieminen2009}, and
Kl\'en--Suomala \cite{KlenSuomala2013}.  The standard metric-space facts used for
bi-Lipschitz comparison, snowflaking, and Hausdorff dimension are taken from
Burago--Burago--Ivanov and Heinonen \cite{BBI2001,Heinonen2001}.

These works show that singular conformal and Grushin-type metrics can
produce nonclassical completion strata, boundary-accessibility effects, and
anomalous Hausdorff dimensions.  The present novelty claim is therefore not
that such phenomena occur in isolation.  Rather, for a finite positive sum
of inverse monomials at a codimension-two corner, the Newton support data
give the exact projective criterion
\(\cH(p,q)<2\min\{p,q\}\), while the diagonal support value
\(A=\cH(1,1)\) simultaneously controls ambient corner accessibility,
normal-fibre collapse, and the exact boundary snowflake exponent.  The
result is a structured classification linking Newton support geometry to
the intrinsic metric completion.

\paragraph{Relation to the author's earlier work.}
The author's related preprint \cite{ZanalAbidinPaper1} studies the
single-interaction and orthogonal-support special cases.  No result from that
manuscript is used as an input here: all metric estimates and completion
arguments are proved independently.  The present contribution is the uniform
classification for arbitrary finite positive Newton supports and the resulting
projective fan and completion geometry.

\paragraph{Claim made here.}
Within the positive finite Newton class \eqref{eq:model}, the paper proves
directly that the homogeneous support function controls resolved
finite-distance accessibility through
\[
\cH(p,q)<2\min\{p,q\},
\]
and that the diagonal order \(A=\cH(1,1)\) controls the intrinsic
completion metric through
\[
d_{\partial}\asymp d_C^{\,1-A/2}.
\]
The mixed two-monomial calculation further shows that the scalar Newton
order must be combined with the differential of the weighted blow-down.
The cited resolution results do not state the metric accessibility or
boundary-snowflake conclusions proved here.  The contribution of the present
paper is therefore the metric-completion theorem attached to positive finite
Newton data; no absolute priority claim beyond the stated comparison is
made.

\paragraph{Excluded claims.}
The paper does not claim a new general resolution-of-singularities theorem,
does not claim a higher-codimension result, and does not claim a
monomialization theorem for arbitrary smooth positive interactions.
Irrational exponents are covered by the metric estimates but not by the
integral rooted/monoidal realization proved here.

\section{Discussion and limitations}\label{sec:limitations}

The finite positive Newton class is broad enough to contain mixed sums of
product singularities while retaining exact control of the dominant
orders.  Positivity is doing real work: it converts the exponent set into a
support function without a separate analysis of zeros of face
polynomials.  Signed interactions should therefore be regarded as a
different problem, requiring nondegeneracy hypotheses on leading face
functions.

The rationality restriction belongs only to the integral geometric
realization.  Discrete quasihomogeneous blow-ups admit an invariant
filtration-based formulation \cite{Behr2021}, while irrational weighted
directions are already covered directly by the metric criterion proved
here.  We do not claim an analogous global integral blow-up realization for
arbitrary irrational exponent data.

The codimension-two assumption is similarly a scope boundary rather than a
negative result.  In higher codimension one expects a support function on
the positive projective cone, but the compatibility and completion
stratification require new arguments.

Finally, no claim is made here for arbitrary smooth positive singular
functions.  The theorem concerns finite positive Newton support, or
functions uniformly comparable to such a support.  Extending existence of
finite resolution to a substantially broader asymptotic class is a separate
problem.

\section{Concluding remarks}\label{sec:conclusion}

Positive finite Newton data provide a direct bridge from resolution
geometry to metric completion.  The support function
\[
\cH(p,q)=\max_j(pa_j+qb_j)
\]
simultaneously records dominance and the sharp finite-distance condition
\[
\cH(p,q)<2\min\{p,q\}.
\]
At the original corner this reduces to
\[
\max_j(a_j+b_j)<2,
\]
while the same diagonal order controls the snowflaked boundary metric and
its Hausdorff dimension.  The mixed benchmark shows why the scalar
coefficient and the blow-down differential must be resolved together.
For rational data, rooted denominator clearing followed by established
monoidal refinement supplies a finite smooth geometric realization.  The
separation between metric, smooth-resolution, and global rational
hypotheses makes the resulting theorem both sharper and more transparent.
Within the positive finite Newton class, the homogeneous Newton support
function is therefore the organizing metric invariant governing
accessibility, completion geometry, and boundary structure.

\end{document}